\documentclass[12pt]{amsart}
\usepackage[all]{xy}
\usepackage{amssymb}
\usepackage{algorithmicx}
\usepackage{tikz}
\usepackage[english]{babel}
\usepackage{graphics}
\usepackage{graphicx}
\graphicspath{ {images/} }
\usepackage{epstopdf}
\DeclareMathOperator{\Ima}{Im}
\DeclareMathOperator{\Ker}{Ker}
\usepackage{amsmath}
\usepackage{cite} 
\calclayout                   
\newtheorem{dummy}{anything}[section]
\newtheorem{theorem}[dummy]{Theorem}

\newtheorem{lemma}[dummy]{Lemma}

\newtheorem{proposition}[dummy]{Proposition}

\theoremstyle{definition}
\newtheorem{definition}[dummy]{Definition}

\newtheorem{example}[dummy]{Example}

\newtheorem{remark}[dummy]{Remark}

\newcommand{\bF}{\mathbf F}

\newcommand{\bR}{\mathbb R}

\def\:{\mkern 1.2mu \colon}
\newcommand{\mmatrix}[4]{\left (\vcenter
{\xymatrix@C-2pc@R-2pc{#1&#2\\#3&#4} } \right )}

               \DeclareMathOperator{\im}{im}

\numberwithin{equation}{section} 

\begin{document}
\title[Homological Properties of Persistent Homology]
{Homological Properties of Persistent Homology}

\subjclass[2010]{Primary:57R70, 37E35; Secondary:57R05}
\keywords{Perfect discrete Morse function, persistent module, Mayer-Vietoris sequence, relative pair.}

\author{Han\.{I}fe Varl{\i}, Ya\u{G}mur Yilmaz, Mehmetc{\.I}k Pamuk}


\address{Department of Mathematics
\newline\indent
Middle East Technical University
\newline\indent
Ankara 06531, Turkey}  
\email{hisal{@}metu.edu.tr, guler.yagmur{@}metu.edu.tr, mpamuk{@}metu.edu.tr}


\date{\today}

\begin{abstract}
We investigate to what extent persistent homology benefits from the properties of the usual homology theory.
		
\end{abstract}

\maketitle

\section{INTRODUCTION}

When one analyzes data obtained from experiments, some geometric structures can be observed as the characteristic of the data but it is not always true that
every data has such a well-organized geometric structure.  Recent researches reveal that some methods of algebraic topology are effective in the characterization 
of complicated data.  This kind of research is called Topological Data Analysis(TDA) and a common mathematical tool in  TDA is to capture topological features is  
persistent homology which is an algebraic method to determine features of a topological space or a data set by using filtration which comes from a suitable function 
on it.

The persistence of general algebraic invariants of topological spaces is not yet well understood.  Moreover, it is not clear that whether the methods to determine 
classical algebraic invariants work for persistent homology. One such method is the Mayer-Vietoris sequence which allows one to study the homology groups of a given space 
in terms of relatively simpler homology groups of its subspaces.  In \cite{di2011mayer}, it is proved that the Mayer-Vietoris sequence for persistent homology is not exact
(actually exact of order $2$, that is $\im \subset \ker $ instead of equality).  In this paper, we first give concrete examples showing this non-exactness.  Then, we show 
that the Mayer-Vietoris sequence is exact if instead one works with the persistent modules.

While investigating the reasons causing non-exactness we work with discrete Morse functions and in particular perfect discrete Morse functions 
(most suitable for combinatorial purposes) as filtering functions.  We observe that one of the reasons for Mayer-Vietoris sequence being 
not exact is classes that born and die through out the filtration.  To minimize such classes we work with perfect discrete Morse functions.  
We show that even one works with such functions Mayer-Vietoris sequence is still not exact (see Example~\ref{example}).

According to \cite{VDM}, there are at least four natural persistent objects that can be derived from a filtered chain complex.  They are absolute persistent (co)homology or relative persistent (co)homology.  We also consider long (exact) sequence of relative pairs for persistence homology.  In this case, we observe that the sequence for persistent homology of a pair is not exact but again exact of order $2$ and show that it becomes exact if one works with the persistent modules.

Before we finish this section let us point out that throughout the paper, we always work with a field coefficient $\bF$ which will be suppressed in our notation.  Many things might go wrong when one replaces the field $\bF$ with a ring, for example we do not have Theorem~\ref{structure}.  Because of this reason we do not consider universal coefficient theorem here. The duality between persistent homology and cohomology with field coefficients is just the duality between vector spaces.  Moreover, the persistent homology and cohomology have the same barcodes \cite{VDM}.  Our conclusions in terms of persistent homology are still true in terms of persistent cohomology.  For the duality between persistent homology and persistent cohomology we refer the reader to \cite{VDM}.



\section{Preliminaries}

In this section, we review some background knowledge and give some preliminary definitions on discrete Morse theory and persistent homology that will be used throughout 
the paper.  For more details we refer the reader to  \cite{zomorodian2005computing},  \cite{forman1} and \cite{HJ}.

\subsection{Discrete Morse Functions}

Let $K$ be a simplicial complex.  A real valued function $f\colon K\rightarrow \bR$ is a discrete Morse function if for any $p$-cell 
$\sigma \in K$, there is at most one $(p+1)$-cell $\tau$ containing $\sigma$ in its boundary such that $f(\tau)\leq f(\sigma)$ and there 
is at most one $(p-1)$-cell $\nu$ contained in the boundary of $\sigma$ such that $f(\nu)\geq f(\sigma)$.  A $p$-cell $\sigma \in K$ 
is a critical $p$-cell of $f$ if $f(\nu)< f(\sigma)< f(\tau)$ for any $(p-1)-$face $\nu$ and $(p+1)-$coface $\tau$ of $\sigma$.  
A cell is regular if it is not critical.

The function $f$ can also be denoted by arrows pointing in the direction of decrease  and in this way form a gradient vector field where 
regular cells appear in disjoint pairs,
$$
V = \{(\sigma\to \tau) \mid \dim \sigma=\dim \tau-1, \sigma<\tau, f(\sigma)\geq f(\tau)\}.
$$
The number of critical $p$-cells of a discrete Morse function $f$ is always greater than or equal to the $p$-th Betti number of 
$K$ \cite[Corollary~3.7]{forman1}.  A discrete Morse function is called perfect (with respect to the given coefficient ring) if the 
number of critical $p$-cells and  the $p$-th Betti number are the same.

Let $f$ be a discrete Morse function on $K$ and $u \in \bR$.  The sublevel complex $K(u)$ of $K$ is defined as 
$$ 
\displaystyle  K(u) = \bigcup_{ f(\beta)\leq u} ~ \bigcup_{\alpha\leq \beta} ~ \alpha. 
$$
In other words, $K(u)$ contains all cells $\beta$ where $f(\beta)\leq u$ with all of their faces and hence is a subcomplex of $K$.   
 

\subsection{Persistent Homology}

Let $X$ be a topological space and $\phi: X\rightarrow \bR$ be a continuous function.  For $u\in \bR$ a sublevel set of $X$ 
obtained by $\phi$ is defined as 
$$
X_u=\{x\in X  |  \phi(x)\leq u \}.
$$
Note that $\mathcal{F}=\{X_u  |  u=1, \ldots, n\}$ provides a filtration of the space $X$. 

For $ u<v \in \bR $, the $k-th$ persistent homology group $H_k ^{u,v}(X)$ is defined as the image of the homomorphism 
$i_k^{u,v}\colon H_k(X_u)\rightarrow H_k(X_v)$ induced by the inclusion $X_u\hookrightarrow X_v$, 
$$
H_k ^{u, v}(X):=\Ima i_k^{u, v}.
$$ 
Equivalently, the persistent homology groups consist of the homology classes of $X_u$ that are still alive at $X_v$.


\subsection{Persistent Modules}

Next, we give the definition of a persistence module.  As notation we use $H^{u}_{k}(X):=H_{k}(X_u; \bF)$.

\begin{definition}({\cite{zomorodian2005computing}, \cite{HJ} })
Let $\emptyset=X_0\subset X_1\subset \ldots \subset X_u\subset \ldots X_n=X$ be a filtration of $X$.
The $k$-th persistence module of the space $X$ is a family of $k$-th homology modules $H^{u}_{k}(X)$, together with module homomorphisms 
$i_k ^{u}\colon H^{u}_{k}(X) \to H^{u+1}_{k}(X)$ induced by the inclusions $i^{u}:X_u\rightarrow X_{u+1}$.  
A persistence module is said to be of finite type if each component module is finitely generated. 
\end{definition}

The $k$-th persistence module can be given the structure of a finitely generated graded module over a polynomial ring $R[x]$
$$
\mathcal{H}_{k}(X)=\bigoplus^{n} _{u=0} H^{u}_{k}(X).
$$ 
The action of $x$ is given by 
$$
x\cdot (m^0, m^1, m^2, \ldots, m^n)= (0, i_k^{0}(m^0), i_k^{1}(m^1), i_k^{2}(m^2), \ldots, i_k^{n}(m^n) ),
$$ 
where $i_k^{n}(m^n)=m^n$, for $m^n \in H^{n}_{k}(X)$.  That is, the action of the polynomial ring connects the homologies 
across different complexes in the filtration of $X$.  Properties of this algebraic object can tell us all the information  
abot the $k$-th persistent homology groups $H_k^{u, v}$.

\begin{remark}
Since we work with a compact space $X$, the $k$-th persistence module of $X$ is always of finite type.  
\end{remark}

By the structure theorem of graded modules over PIDs \cite{HJ}, we have 

\begin{theorem}{\cite{zomorodian2005computing}} \label{structure}
Suppose $\mathcal{H}_{k}(X)$ is over the polynomial ring $\bF[x]$, where $\bF$ is a field.  Then 
$$
\mathcal{H}_{k}(X)=(\bigoplus _{i} (x^{a_i}))\oplus (\bigoplus _{j} (x^{b_j})/(x^{c_j}))
$$ 
where the sums range over $1\leq i\leq M$ and $1\leq j\leq N$ for non-negative integers $M, N$ and 
$a_i, b_j, c_j$ are non-negative integer powers of $x$.
\end{theorem}

The powers in the above expression captures topological features throughout the filtration. 
They represent the indices in the filtration sequence for which either a new $k$-dimensional 
hole arises or for which a previously existing $k$-dimensional hole, that appear before disappears. The free component gives the features that appear at a certain index but do not disappear as the complex grows. The torsion component encodes the holes that appear for a while and then disappear at a later index.


\section{The Mayer-Vietoris Formula for Persistent Homology}

In this section, we recall that the Mayer-Vietoris sequence for persistent homology groups is not exact but it is of order two \cite{di2011mayer}.
We give a concrete example showing that the Mayer-Vietoris sequence is not exact.  We also show that if one works with graded persistence modules 
instead of persistent homology groups, then the Mayer-Vietoris sequence is exact.

Recall that the Mayer-Vietoris sequence allows one to compute the homology of a complicated topological space $X$ in terms of relatively simpler 
subspaces $A\subset X$, $ B\subset X$ where $X$ is the union of the interiors of $A$ and $B$.  For such a pair of subspaces $A, B \subset X$, 
the Mayer-Vietoris exact sequence has the following form
$$
\displaystyle
\xymatrix @-1pc 
{ &{\cdots}\ar[r] & H_{k+1}(X)\ar[r]^-{\delta_k} & H_{k}(A\cap B) \ar[r]^-{\alpha_k} & H_k(A) \oplus H_k(B)\ar[r]^-{\beta_k}&  H_k(X) \ar[r] & {\cdots}\ar[r]& H_0(X) \ar[r] & 0}
$$
where 
$$ 
\delta_k([x])=[\partial(x|_A)], \  \alpha_k([y])=([y], [-y]) \ \textrm{and} \  \beta_k(([z], [z']))=[z+z'].
$$ 

In the case of very large $X$, it may often be difficult to apply persistent homology algorithm to all of $X$ and one may instead find it easier to 
work with relatively simpler subspaces.  A Mayer-Vietoris sequence would increase our ability in terms of visualization by partitioning a set of data 
into number of parts which are easily recognizable.

Let $X_u, A_u, B_u$ and $(A\cap B)_u$ denote the sublevel sets as in \cite{di2011mayer}, and  consider the following diagram:
\begin{equation}
\begin{split}
\xymatrix @-0,5pc  
{
\cdots \ar[r] & H_{k+1}(X_u)\ar[r]^-{\delta^u_k} \ar[d]^-{h_{k+1}} & H_k((A \cap B)_u)\ar[r]^-{\alpha^u_k}\ar[d]^-{f_k} & H_k(A_u)\oplus H_k(B_u)\ar[r]^-{\beta^u_k}\ar[d]^-{g_k} & 
H_{k}(X_u)\ar[r]\ar[d]^-{h_k}& \cdots \\
\cdots \ar[r] & H_{k+1}(X_v)\ar[r]^-{\delta^v_k} & H_{k}((A \cap B)_v) \ar[r]^-{\alpha^v_k} & H_{k}(A_v)\oplus H_{k}(B_v)\ar[r]^-{\beta^v_k} & 
H_{k}(X_v)\ar[r] & \cdots
} 
\end{split}
\end{equation}
\noindent
where the horizontal lines belong to the usual Mayer-Vietoris sequence of the triads $(X_u, A_u, B_u)$, and $(X_v, A_v, B_v)$.  
The horizontal homomorphisms are defined as $\delta^u_k([x])=[\partial(x|_A)]$, $\alpha^u_k([y])=([y], -[y])$ and $\beta^u_k(([z], [z']))=[z+z']$. 

\begin{lemma}\cite{di2011mayer}\label{exactness}
Each horizontal line in above diagram is exact.  Moreover, each square in the same diagram is commutative.
\end{lemma}

Note that by definition, we have for every $k \in \mathbb{Z}$, and every $u<v \in \bR$,
\begin{enumerate}
\item $\im f_k = H^{u,v}_{k}(A \cap B)$,
\item $\im g_k = H^{u,v}_{k}(A) \oplus H^{u,v}_{k}(B)$,
\item $\im h_k = H^{u,v}_{k}(X)$.
\end{enumerate}

The following proposition states that the commutativity of squares in above diagram induces a Mayer-Vietoris sequence of order $2$ involving 
the $k$-th persistent homology groups for every integer $k$.

\begin{proposition}\cite{di2011mayer}	
Let us consider the sequence of homomorphisms of persistent homology groups
\begin{equation}
\displaystyle
\xymatrix @-1pc 
{ &{\cdots}\ar[r] & H^{u,v}_{k+1}(X)\ar[r]^-{\delta} & H^{u,v}_{k}(A\cap B) \ar[r]^-{\alpha} & H^{u,v}_{k}(A) \oplus H^{u,v}_{k}(B)\ar[r]^-{\beta}& H^{u,v}_{k}(X) \ar[r] & {\cdots}& }
\end{equation}
\noindent
where $\delta=\delta_k^{v}|_{\im h_{k+1}}$, $\alpha=\alpha_k^{v}|_{\im f_{k}}$, and $\beta=\beta_k^{v}|_{\im g_{k}}$, is of order $2$.  That is, 
\begin{enumerate}
	\item $ \im \delta \subseteq \ker \alpha $,
	\item $ \im \alpha \subseteq \ker \beta$,
	\item $ \im \beta \subseteq \ker \delta$.
\end{enumerate}	
\end{proposition}

Although it is given that $\im \delta \subseteq \ker\alpha$, by diagram chasing one can easily show that $\ker\alpha \subseteq \im \delta $ is not always the case.
For example, let $c \in \ker \alpha \subseteq H^{u,v}_{k}(A \cap B)$.  By definition $c \in H_k((A \cap B)_v)$ and also it has to be 
the image of a class in $H_k((A \cap B)_u)$, abusing the notation we denote this class also with $c$. Moreover, $\alpha = {\alpha}^v_k |_{\im f_k}$
implies that $c \in \ker {\alpha}^v_k $.  So, there exists a $d \in H_{k+1}(X_v)$ such that ${\delta}^v_k (d) = c$ since the second row of the above 
diagram is exact.  On the other hand, $c \in H^{u,v}_{k}(A \cap B)$ implies $c \in H_k((A \cap B)_u)$.  But if ${\alpha}^u_k (c) \neq 0$, then we 
cannot be sure that there exist a $d \in H_{k+1}(X_u) $ such that $d \in H^{u,v}_{k+1}(X) $.  Therefore, $c$ may not be in the $\im \delta$. 
Let us further explain this situation through the following example.

\begin{example}\label{example}
For a given complex $X$, a filtration is defined as a nested sequence of its subcomplexes often generated by a  filtering function, which in our example is a 
perfect discrete Morse function.

Consider the triangulation of the torus with a perfect discrete Morse function on it.  The subspaces $A$ and $B$ are shown in Figure \ref{fig1}.  
\begin{figure}[ht]
\centering
\includegraphics[scale=0.60]{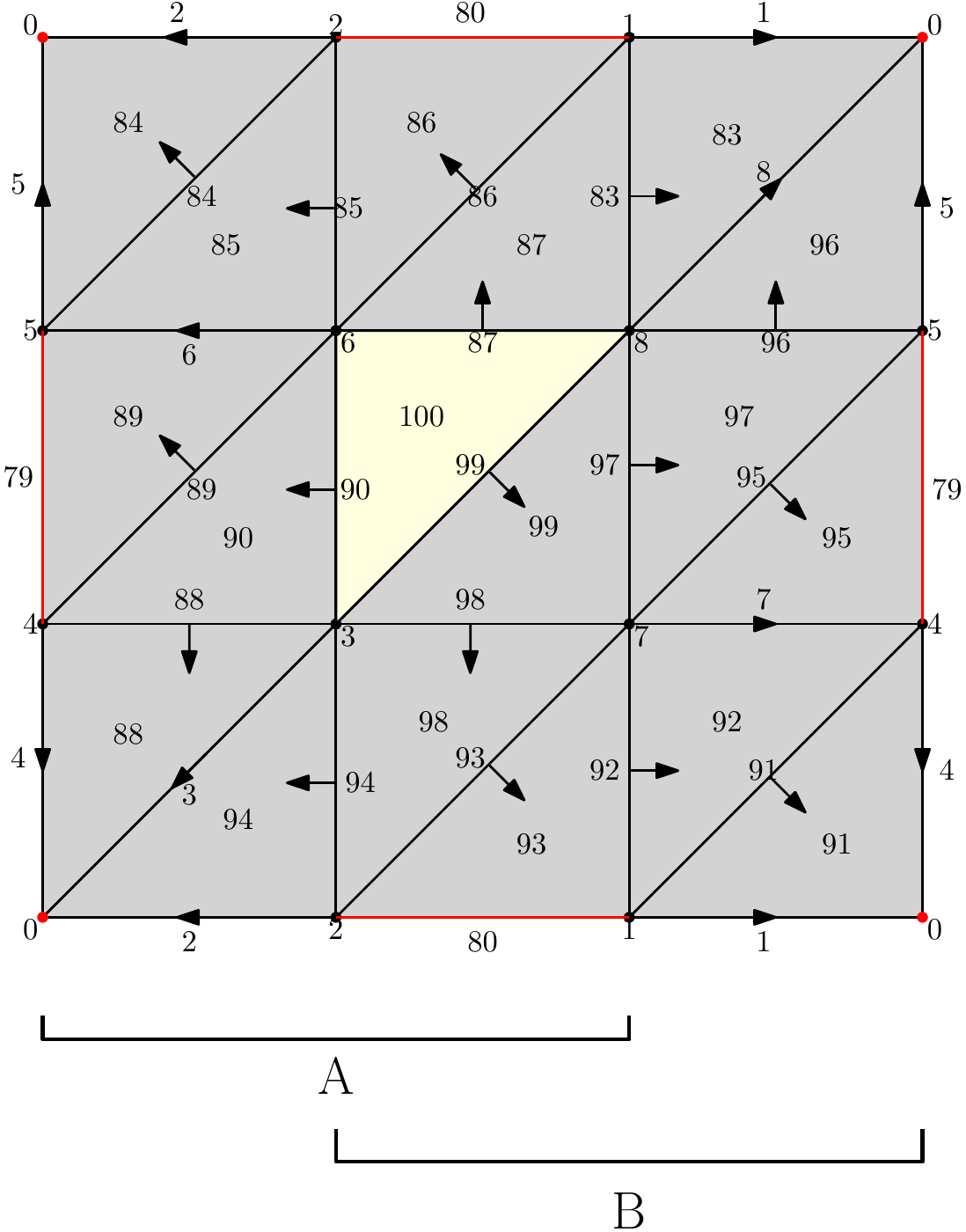}
\caption{Perfect discrete Morse function on the triangulation of a torus.}
\label{fig1}
\end{figure} 
The perfect discrete Morse function gives a filtration for $X$, $A$, $B$ and $A \cap B$ as follows:
\begin{align*}
&\emptyset \subseteq X_0 \subseteq X_6 \subseteq X_8 \subseteq X_{79} \subseteq X_{95} \subseteq X_{100} = X.\\&
\emptyset \subseteq A_0 \subseteq A_6 \subseteq A_8 \subseteq A_{79} \subseteq A_{95} \subseteq A_{100} = A.&\\&
\emptyset \subseteq B_0 \subseteq B_6 \subseteq B_8 \subseteq B_{79} \subseteq B_{95} \subseteq B_{100} = B.\\&
\emptyset \subseteq (A \cap B)_0 \subseteq (A \cap B)_6 \subseteq (A \cap B)_8 \subseteq(A \cap B)_{79} \subseteq (A \cap B)_{95} \subseteq (A \cap B)_{100} = A \cap B.
\end{align*}

Note that the indices in the filtration comes from the values of given perfect discrete Morse function. See Figure \ref{fig3} for the filtration level $95$.
\begin{figure}[ht]
\centering
\includegraphics[scale=0.70]{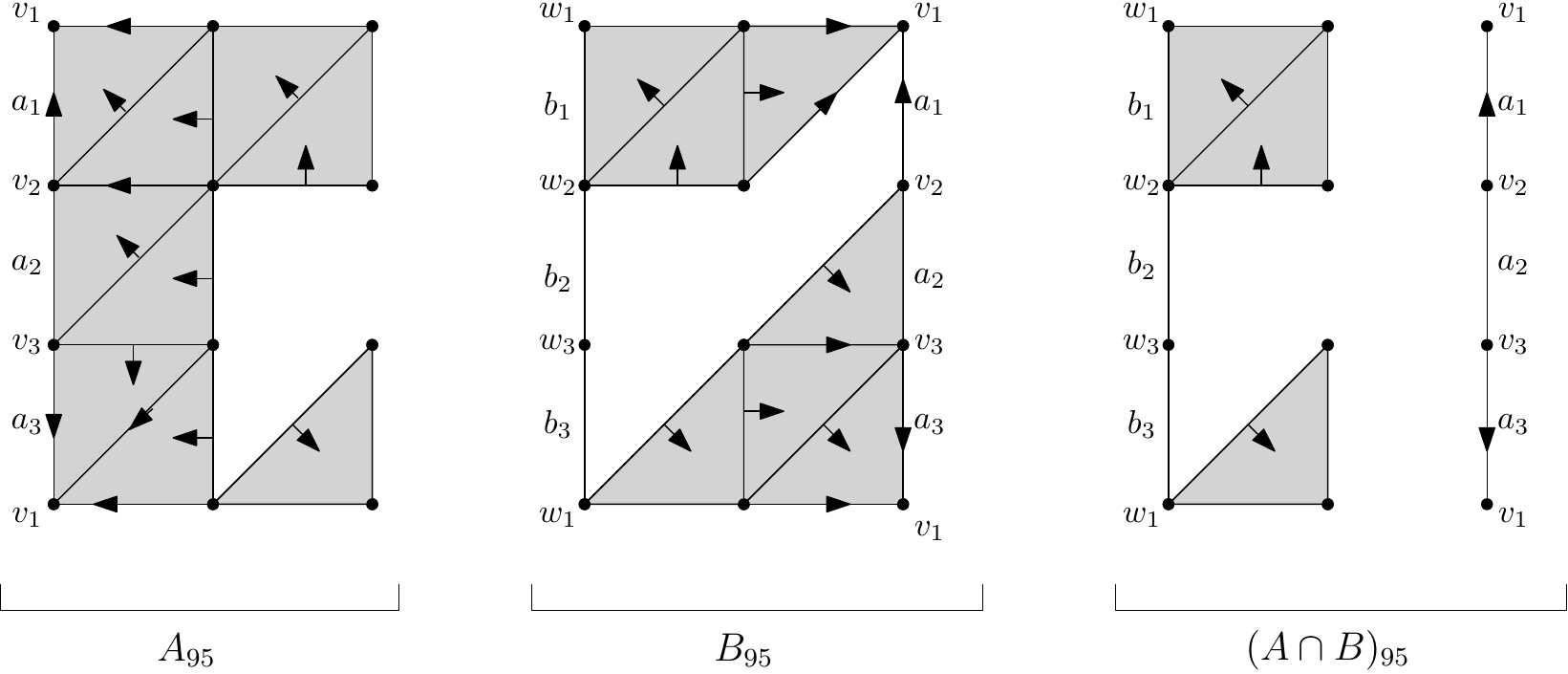}
\caption{}
\label{fig3}
\end{figure}
	
Let us consider the Morse chain complex for the level $B_{95}$ in figure \ref{fig3}:
$$
\displaystyle
\xymatrix @-0.7pc 
{ & 0\ar[r] & C_2\ar[r]^-{\partial_2} & C_1\ar[r]^-{\partial_1} & C_0\ar[r]& {0}& }
$$
	
By using Morse homology, we can induce the boundary maps as follows: 
$\partial_1(b_1)=w_1 - w_2\neq 0$, $\partial_1(b_2)=w_2 - w_3\neq 0$, $\partial_1(b_3)=w_3 - w_1\neq 0$ but $\partial_1(b_1+b_2+b_3)=0$.  
Therefore $b_1+b_2+b_3$ is a generator of $H_1(B_{95})$ call it $\gamma$.  
Similarly, $a_2$ is also a generator of $H_1(B_{95})$ call it $\alpha$.
	
After applying the same process we get:

\begin{eqnarray*}
 H_1(A_{95})&=&\langle\alpha\rangle,\\ H_1(A_{100})&=& \langle\alpha\rangle,\\
 H_1(B_{95})&=& \langle\alpha, \gamma \rangle,\\ H_1(B_{100}) &=& \langle \alpha \rangle,\\
  H_1((A \cap B)_{95}) &=& \langle \alpha, \gamma \rangle,\\ H_1((A \cap B)_{100})& =& \langle \alpha, \gamma \rangle.
\end{eqnarray*}
By the definition of persistent homology group we have,	
\begin{eqnarray*}
 H^{95,100}_1(A) &=& \langle \alpha \rangle,\\
 H^{95,100}_1(B)&=&\langle \alpha \rangle,\\
H^{95,100}_1(A \cap B) &=& \langle \alpha, \gamma \rangle
\end{eqnarray*}
	
Now, consider the Mayer-Vietoris sequence for persistent homology:
$$
\displaystyle
\xymatrix @-1pc 
{ &{0}\ar[r]^-{\delta} & H^{95,100}_{1}(A \cap B)\ar[r]^-{\alpha} & H^{95,100}_{1}(A) \oplus H^{95,100}_{1}(B)\ar[r]^-{\beta}& H^{95,100}_{1}(X) \ar[r] & {\cdots}& }
$$
Clearly, $\im \delta_1 = 0$. Also $\ker \alpha_1 = \langle \gamma \rangle \neq 0$.  Therefore, $\Ima \delta_1 \neq \ker \alpha_1$ which implies this sequence is not exact.
\end{example}

Next, we show that if instead one works with the persistent modules the Mayer-Vietoris sequence for graded persistence modules is exact.  Before we state the theorem, 
let us define
$$
s^u_k\colon H_k((A \cap B)_u)\to H_k(A_u) \ \textrm{defined by} \ s^u_k([y])=[y]
$$
and also 
$$
t^u_k\colon H_k((A \cap B)_u)\rightarrow H_k(B_u)  \ \textrm{defined by} \ t^u_k([y])=[-y]
$$ 
so that $\alpha^u_k([y])=(s^u_k([y]),t^u_k([y]))$.

\begin{theorem}\label{Mayermodule}
The Mayer-Vietoris sequence of graded persistence modules 

\begin{equation}
\begin{split}
\displaystyle
\xymatrix @-1pc 
{ &{\cdots}\ar[r] & \mathcal{H}_{k+1}(X)\ar[r]^-{\delta} & \mathcal{H}_{k}(A\cap B) \ar[r]^-{\alpha} & 
\mathcal{H}_{k}(A) \oplus \mathcal{H}_{k}(B)\ar[r]^-{\beta}& \mathcal{H}_{k}(X) \ar[r] & {\cdots}& }
\end{split}
\end{equation}

where 
\begin{align}
\delta(m^0, m^1, \ldots, m^n )&=(\delta_k^{0}(m^0), \delta_k^{1}(m^1), \ldots, \delta_k^{n}(m^n)),\nonumber
\\
 \alpha(y^0, y^1, \ldots, y^n)&=((s_k^{0}(y^0), s_k^{1}(y^1), \ldots, s_k^{n}(y^n)), (t_k^{0}(y^0), t_k^{1}(y^1), \ldots, t_k^{n}(y^n))),\nonumber
\\
 \beta(a, b)&=(\beta_k^{0}(a^0, b^0), \beta_k^{1}(a^1, b^1), \ldots, \beta_k^{n}(a^n,b^n)),\nonumber 
 \end{align}
for all  $m^u \in H^{u}_{k+1}(X)$,  $y^u \in H^{u}_{k}(A\cap B)$, $a^u \in H^{u}_{k}(A)$ and $b^u \in H^{u}_{k}(B)$,  is exact. 
That is, 
\begin{enumerate}
\item $ \im \delta = \ker \alpha $,
\item $ \im \alpha = \ker \beta$,
\item $ \im \beta = \ker \delta$.
\end{enumerate}
\end{theorem}

\begin{proof}
We only prove  $(i)$.  The other claims can be proven analogously.  

Let $(y^0, y^1, \ldots, y^n)\in \im \delta$.  Then there exist an element $(m^0, m^1, \ldots, m^n)\in \mathcal{H}_{k+1}(X)$ such that 
$$
\delta(m^0, m^1, \ldots, m^n)=(\delta_k^{0}(m^0), \delta_k^{1}(m^1), \ldots, \delta_k^{n}(m^n))=(y^0, y^1, \ldots, y^n).
$$  
In particular, $\delta_k^{u}(m^u)=y^u$ and  $y^u\in\im\delta_k^{u}=\ker\alpha_k^{u}$ for all $u$ by Lemma~\ref{exactness}.  
Then we get 
$\alpha_k^{u}(y^u)=(s_k^{u}(y^u), t_k^{u}(y^u))=(0,0)$ and $s_k^{u}(y^u)=t_k^{u}(y^u)=0$.  
Therefore, we have 
\begin{align}
\alpha(y^0, y^1, \ldots, y^n)&=((s_k^{0}(y^0), s_k^{1}(y^1), \ldots, s_k^{n}(y^n)), (t_k^{0}(y^0), t_k^{1}(y^1), \ldots, t_k^{n}(y^n)))\nonumber\\
&=((0,0,\ldots, 0), (0,0,\ldots, 0)) \nonumber
\end{align}
and $(y^0, y^1, \ldots, y^n)\in \ker \alpha$.  Hence we get that $\im \delta \subseteq \ker \alpha$.

Let $(c^0, c^1, \ldots, c^n )\in \ker \alpha$.
Then 
\begin{align}
\alpha(c^0, c^1, \ldots, c^n)&=((s_k^{0}(c^0), s_k^{1}(c^1), \ldots, s_k^{n}(c^n)), (t_k^{0}(c^0), t_k^{1}(c^1), \ldots, t_k^{n}(c^n))) \nonumber\\
 &=((0,0, \ldots, 0),(0,0, \ldots, 0))\nonumber.
\end{align}
By equalities of the tuples,  we get $s_k^{u}(c^u)=0$ and $t_k^{u}(c^u)=0$ for all $c^u\in H^{u}_{k}(A\cap B)$ 
Thus, $c^u\in \Ker \alpha_k^{u}=\Ima\delta_k^{u}$ by Lemma~\ref{exactness}.  Then there exists $m^u\in H^{u}_{k+1}(X)$ such that $\delta_k^{u}(m^u)=c^u$ for all $u\geq 0$ 
and we get 
$$
(c^0, c^1, \ldots, c^n)=(\delta_k^{0}(m^0),\delta_k^{1}(m^1), \ldots, \delta_k^{n}(m^n))=
 \delta(m^0, m^1, \ldots, m^n).
$$ 
Therefore, $(c^0, c^1, \ldots, c^2)\in \im \delta$ which implies that $\ker \alpha \subseteq \im \delta $. 
\end{proof}

Now, let us briefly see how the sequence that is just constructed gives an exact sequence for the Example~\ref{example}.  Recall that in Example~\ref{example}, 
the reason for non-exactness is that there is a $\gamma \in H^{95,100}_{1}(A \cap B)$ that belongs to $\ker \alpha$ but not in $\im \delta$.  
In terms of persistence modules $\gamma$ corresponds to $(0, 0, 0, 0, \gamma, 0)\in \mathcal{H}_1(A \cap B)$ and 
$\alpha(0, 0, 0, 0, \gamma, 0)=((0, 0, 0, 0, 0, 0), (0, 0, 0, 0, \gamma, 0)) \in \mathcal{H}_1(A) \bigoplus \mathcal{H}_1(B)$.  Hence it is not in $\ker \alpha$.  
Moreover $x \cdot (0, 0, 0, 0, \gamma, 0)= (0, 0, 0, 0, 0, \gamma)$ and  $\alpha(0, 0, 0, 0, 0, \gamma)=((0, 0, 0, 0, 0, 0), (0, 0, 0, 0, 0, 0))$.  But in this case 
$(0, 0, 0, 0, 0, \sigma)\in \mathcal{H}_2(X)$ and $\delta (0, 0, 0, 0, 0, \sigma)=(0, 0, 0, 0, 0, \gamma)$.


\section{Long (exact) sequence for persistence homology}

In this section, we show that the long sequence for persistence homology groups of a pair $(X, A)$, where $X$ is a compact, 
triangulated space and $A\subset X$, is not exact but exact of order $2$.  We give a concrete example which shows 
that the long sequence for the relative pair is not exact.  We also prove that the long sequence for the graded 
persistence modules of the pair $(X, A)$ is exact.

For any  $u <v \in \bR$, we can consider the following diagram:
\begin{equation}
\begin{split}
\displaystyle
\xymatrix @-0,5pc  
{ 
	\cdots \ar[r] & H_{k+1}(X_u,A_u)\ar[r]^-{\delta^u_k} \ar[d]^-{h_{k+1}} & H_k(A_u)\ar[r]^-{\alpha^u_k}\ar[d]^-{f_k} & H_k(X_u)\ar[r]^-{\beta^u_k}\ar[d]^-{g_k} & 
	H_{k}(X_u,A_u)\ar[r]\ar[d]^-{h_k}& \cdots \\
	\cdots \ar[r] & H_{k+1}(X_v, A_v)\ar[r]^-{\delta^v_k}& H_{k}(A_v) \ar[r]^-{\alpha^v_k} & H_{k}(X_v)\ar[r]^-{\beta^v_k} & 
	H_{k}(X_v, A_v)\ar[r] & \cdots \\
	}
\end{split}
\end{equation}

\noindent
where $\delta^u_k([c])= [(\partial c)|_{A_u}]$ and $\alpha^u_k$,  $\beta^u_k$ are the homomorphisms induced by inclusion and quotient map, respectively.

\begin{lemma}\label{comd}
Each horizontal line in the above diagram is exact.  Moreover, each square is commutative.
\end{lemma}

\begin{proof}
This follows from the exactness of the rrelative homology sequence of the pairs $(X_u, A_u)$ and $(X_v, A_v)$ and naturality of the diagram.
\end{proof}

\begin{remark}
Note that, for every $u<v\in \bR$,  $(X_u, A_u)\subset(X_v, A_v)$ gives a filtration of $(X, A)$ since $X_u\subset X_v$ and $A_u\subset A_v$.  We denote 
$(X, A)_u:=(X_u, A_u)$ and $H^{u}_{k}((X, A)):=H_{k}((X, A)_u):=H_{k}(X_u, A_u)$.  By definition of persistent homology groups, we also have 
\begin{enumerate}
\item $\im f_k = H^{u,v}_{k}(A)$
\item $\im g_k = H^{u,v}_{k}(X)$
\item $\im h_k = H^{u,v}_{k}(X,A)$
\end{enumerate}

\end{remark}

\begin{proposition}\label{relativemayer}
	
The following sequence of homomorphisms of persistent homology groups
$$
\displaystyle
\xymatrix @-0.7pc 
{ &{\cdots}\ar[r] & H^{u,v}_{k+1}(X,A)\ar[r]^-{\delta} & H^{u,v}_{k}(A) \ar[r]^-{\alpha} & H^{u,v}_{k}(X) \ar[r]^-{\beta}& H^{u,v}_{k}(X,A) \ar[r] & {\cdots}& }
$$
where $\delta=\delta_k^{v}|_{\im h_{k+1}}$, $\alpha=\alpha_k^{v}|_{\im f_{k}}$, and $\beta=\beta_k^{v}|_{\im g_{k}}$, is exact of order $2$.  That is, 
\begin{enumerate}
\item $ \im \delta \subseteq \ker \alpha $,
\item $ \im \alpha \subseteq \ker \beta$,
\item $ \im \beta \subseteq \ker \delta$.
\end{enumerate}	

\end{proposition}

\begin{proof}
First observe that, by Lemma~\ref{comd}, $\im \delta\subset \im f_k$, $\im \alpha\subset \im g_k$ and $\im \beta\subset \im h_k$.
We prove only claim $(i)$.  The claims $(ii)$ and $(iii)$ cam be deduced similarly.

Let $\sigma \in \im \delta$.  Then there exists $\gamma \in H_{k+1}(X_u, A_u)$ such that $\delta(\gamma)=\delta_k^{v}(h_{k+1}(\gamma))=\sigma$.
Thus we get $\sigma \in \im \delta_k^{v}=\ker \alpha_k^{v}$.
Since $\sigma \in \im f_k$, there exists $\tau \in H_k(a_u)$ such that $f_k(\tau)=\sigma.$
Then 
$$
0=\alpha_k^{v}(\sigma)=\alpha_k^{v}(f_{k}(\tau)=\alpha_k^{v}|_{\im f_{k}}(\sigma)=\alpha(\sigma).
$$
Therefore, $\sigma \in \ker\alpha$
\end{proof}

The above proposition shows that the sequence of homomorphisms of persistence homology group is exact of order two.  
Now we give an example showing that this sequence is not exact.

\begin{example}
Let $X$ be the orientable closed surface of genus $2$ and $A$ be a separating circle on $X$ as in Figure \ref{fig:relativePH}.  
Let $X_u$, $X_v$ and $A_u$, $A_v$ be the sublevel sets obtained by the height function $\varphi:X\rightarrow \bR$ for $u, v \in \bR$ such that $u<v$.
We have, $H_*(X, A)\cong H_*(X/A)$ and $H_*(X_u, A_u)\cong H_*(X_u/A_u)$ which in turn imply 
\begin{eqnarray*}
H_2(X_u)&=&0=H_2(X_u, A_u)=H_2(A_u)=H_2(A_v),\\ H_1(X_u)&\cong&\langle a, b, c, d, e \rangle,\\ H_1(A_u)&\cong&\langle A \rangle,\\ H_1(X_u,A_u)&\cong&\langle a,b,c,d \rangle\\
H_1(X_v)&\cong&\langle a, b, c, d \rangle=H_1(X_v, A_v),\\ H_1(A_v)&\cong&\langle A \rangle.
\end{eqnarray*}
Note that, although the homology class $\langle A \rangle$ in $H_1(A_u)$ lives in $H_1(X_u)$ as $a+e$, it is null-homologous in $H_1(X_v)$.

That is 
\begin{enumerate}
\item $\im h_2 = H^{u,v}_{2}(X,A)\cong 0$
\item $\im f_1 = H^{u,v}_{1}(A)\cong \langle A \rangle$
\item $\im g_1 = H^{u,v}_{1}(X)\cong\langle a,b,c,d \rangle$
\item $\im h_1 = H^{u,v}_{1}(X,A)\cong \langle a,b,c,d \rangle$
\end{enumerate}

Therefore, the sequence of homomorphisms of the corresponding persistence homology groups  
$$
\displaystyle
\xymatrix @-1pc 
{ &{0}\ar[r] & H^{u, v}_{2}(X, A)\ar[r]^-{\delta} & H^{u, v}_{1}(A) \ar[r]^-{\alpha} & 
H^{u, v}_{1}(X) \ar[r]^-{\beta}& H^{u, v}_{1}(X, A) \ar[r] & {0}& }
$$

is obtained as in the following :

$$
\displaystyle
\xymatrix @-1pc 
{ &{0}\ar[r] & {0}\ar[r]^-{\delta} & \langle A \rangle \ar[r]^-{\alpha} & \langle a, b, c, d \rangle \ar[r]^-{\beta}& \langle a, b, c, d \rangle \ar[r] & {\cdots}& }
$$

By the definition of the homomorphisms, $\alpha([A])=0$. Then $[A]\in \ker \alpha$ but $\im \delta=0$.  

Thus $\ker\alpha\nsubseteq \im\delta$ and the long sequence for the persistence homology groups of relative pair is not exact.
\begin{figure}[hbtp]
			\centering
			\hspace{1.5 cm}
			\includegraphics[width=.9\textwidth,height=.8\textheight]{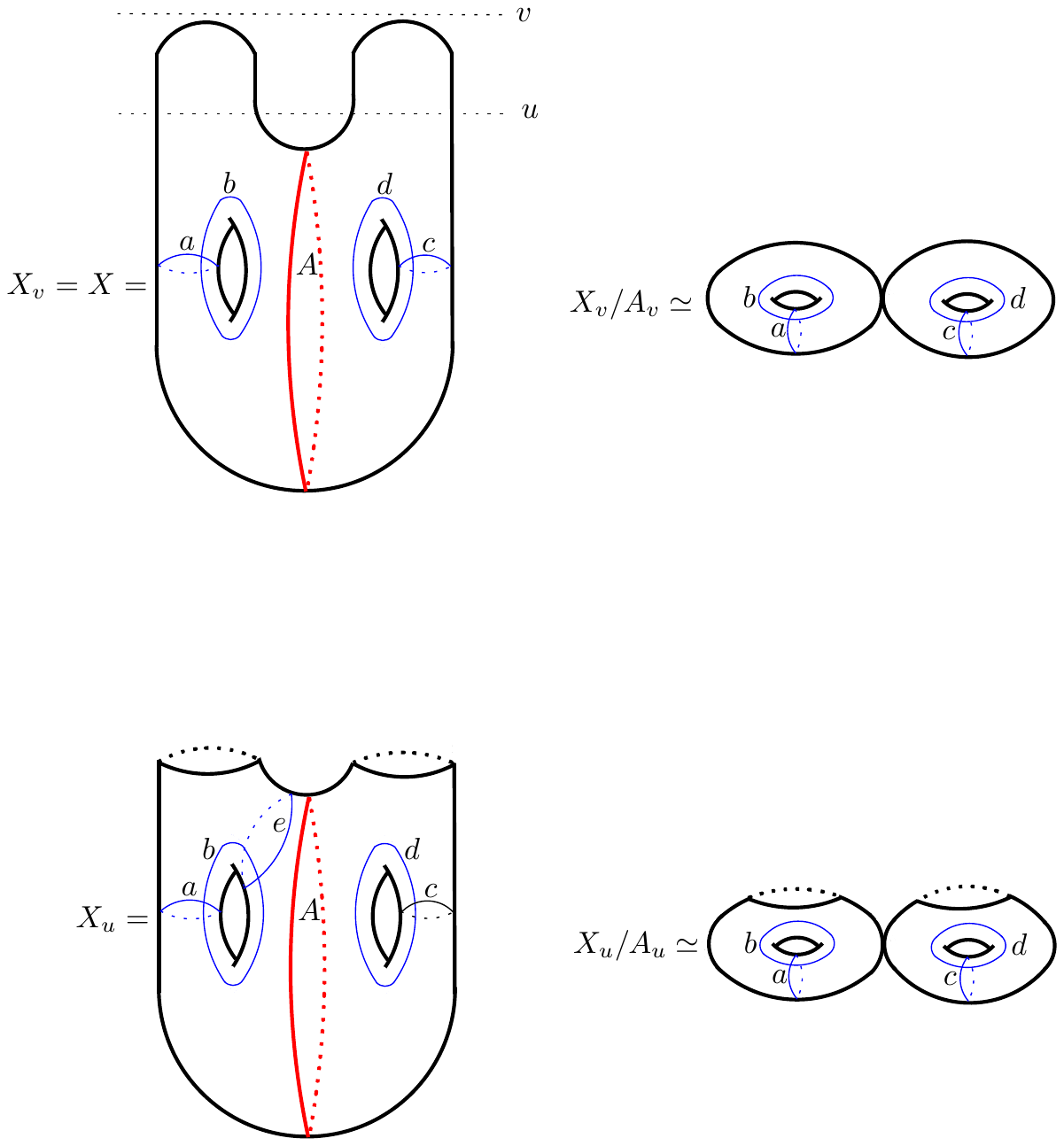}
			\caption{}
			\label{fig:relativePH}
\end{figure}

\end{example}

Next, we prove that the long sequence for graded persistence modules of a pair $(X, A)$ is exact.  Let 
$$
\mathcal{H}_{k}(X, A)=\bigoplus^{n} _{u=0 } H^{u}_{k}(X, A), \  \mathcal{H}_{k}(X)=\bigoplus^{n} _{u=0 } H^{u}_{k}(X) \  \textrm{and} \  
\mathcal{H}_{k}(A)=\bigoplus^{n} _{u=0 } H^{u}_{k}(A)
$$
be the graded persistence modules of $(X,A)$, $X$ and $A$, respectively.
\begin{theorem}
The long sequence of graded persistence modules of a relative pair $(X, A)$  
$$
\displaystyle
\xymatrix @-1pc 
{ &{\cdots}\ar[r] & \mathcal{H}_{k+1}(X,A)\ar[r]^-{\delta} & \mathcal{H}_{k}(A) \ar[r]^-{\alpha} & \mathcal{H}_{k}(X)\ar[r]^-{\beta}& \mathcal{H}_{k}(X, A) \ar[r] & {\cdots}& }
$$
is exact, where 
\begin{align}
\delta(m^0, m^1, \ldots, m^n)&=(\delta_k^{0}(m^0), \delta_k^{1}(m^1), \ldots, \delta_k^{n}(m^n)), \nonumber
\\
 \alpha(y^0, y^1, \ldots, y^n)&=(\alpha_k^{0}(y^0), \alpha_k^{1}(y^1), \ldots, \alpha_k^{n}(y^n)),\nonumber
\\
 \beta(x^0, x^1, \ldots, x^n)&=(\beta_k^{0}(x^0), \beta_k^{1}(x^1), \ldots, \beta_k^{n}(x^n)).\nonumber
 \end{align}

Equivalently we have
\begin{enumerate}
\item $ \im \delta = \ker \alpha $,
\item $ \im \alpha = \ker \beta$,
\item $ \im \beta = \ker \delta$.
\end{enumerate}

\end{theorem}

\begin{proof}
We prove only claim $(i)$.  Claims $(ii)$ and $(iii)$ can be obtained analogously.  

Let $(y^0, y^1, \ldots, y^n)\in \im \delta$.  Then there exists $ (m^0, m^1, \ldots, m^n)\in \mathcal{H}_{k+1}(X)$ such that 
$$
\delta(m^0, m^1, \ldots, m^n)=(\delta_k^{0}(m^0), \delta_k^{1}(m^1), \ldots, \delta_k^{n}(m^n))=(y^0, y^1, \ldots, y^n).
$$  
So $\delta_k^{u}(m^u)=y^u$ and  $y^u\in\im \delta_k^{u}=\ker\alpha_k^{u}$ for all $u$ by Lemma \ref{comd}.  
Then we get $\alpha_k^{u}(y^u)=0$.  Therefore, 
\begin{align}
\alpha(y^0, y^1, \ldots, y^n)&=(\alpha_k^{0}(y^0), \alpha_k^{1}(y^1), \ldots, \alpha_k^{n}(y^n))\nonumber\\
&=(0,0, \ldots, 0) \nonumber
\end{align}
and $(y^0, y^1, \ldots, y^n)\in \ker \alpha$.  Thus, $ \im \delta \subset \ker \alpha $.

Let $(c^0,  c^1, \ldots, c^n)\in \ker \alpha$.
Then 
\begin{align}
\alpha(c^0, c^1,  \ldots, c^n)&=(\alpha_k^{0}(c^0), \alpha_k^{1}(c^1), \ldots, \alpha_k^{n}(c^n)) \nonumber\\
 &=(0, 0, \ldots, 0)\nonumber.
\end{align}

We get $\alpha_k^{u}(c^u)=0$ for all $c^u\in H^{u}_{k}(A)$.  Thus, $c^u\in \ker \alpha_k^{u}=\im\delta_k^{u}$ by Lemma \ref{comd}.  
Then there exists $m^u\in H^{u}_{k+1}(X, A)$ such that $\delta_k^{u}(m^u)=c^u$ for all $u\geq 0$ and then we get 
$$
(c^0, c^1,  \ldots, c^n)=(\delta_k^{0}(m^0), \delta_k^{1}(m^1), \ldots, \delta_k^{n}(m^n))=\delta(m^0, m^1, \ldots, m^n).
$$ 
Therefore, $(c^0, c^1, \ldots, c^n)\in \im \delta$ and this implies that $ \ker \alpha \subset \im \delta$. 
\end{proof}
	

\bibliographystyle{amsplain}
\providecommand{\bysame}{\leavevmode\hbox
to3em{\hrulefill}\thinspace}
\providecommand{\MR}{\relax\ifhmode\unskip\space\fi MR }
 \MRhref  \MR
\providecommand{\MRhref}[2]{
  \href{http://www.ams.org/mathscinet-getitem?mr=#1}{#2}
 } \providecommand{\href}[2]{#2}

\end{document}